\documentclass[10pt, twoside, reqno]{article}
\usepackage{amssymb,amsthm}

\newtheorem{thm}{Theorem}[section]
\newtheorem{theorem}[thm]{Theorem}
\newtheorem{definition}[thm]{Definition}
\newtheorem{example}[thm]{Example}
\newtheorem{remark}[thm]{Remark}
\newtheorem{question}[thm]{Question}
\newtheorem{lemma}[thm]{Lemma}
\newtheorem{corollary}[thm]{Corollary}

\def\cl{\mathrm{cl}}
\def\Int{\mathrm{Int}}
\def\cprime{$'$}

\begin{document}

\title{The non-Urysohn number of a topological space}
\author{Ivan S. Gotchev\\
Department of Mathematical Sciences,\\
Central Connecticut State University,\\
New Britain, CT 06050, USA\\
E-mail: gotchevi@ccsu.edu}

\date{}

\maketitle


\renewcommand{\thefootnote}{}



\renewcommand{\thefootnote}{\arabic{footnote}}
\setcounter{footnote}{0}

\begin{abstract}
We call a nonempty subset $A$ of a topological space $X$  
\emph{finitely non-Urysohn} if for every nonempty finite subset $F$ of $A$ and 
every family $\{U_x:x\in F\}$ of open neighborhoods $U_x$ of $x\in F$, 
$\cap\{\mathrm{cl}(U_x):x\in F\}\ne\emptyset$ and 
we define \emph{the non-Urysohn number of $X$} 
as follows: $nu(X):=1+\sup\{|A|:A$ is a finitely non-Urysohn subset of 
$X\}$.\\
Then for any topological space $X$ and any subset $A$ of $X$ we prove 
the following inequalities: (1) $|\cl_\theta(A)|\le |A|^{\kappa(X)}\cdot nu(X)$,
(2) $|[A]_\theta|\le (|A|\cdot nu(X))^{\kappa(X)}$, 
(3) $|X|\le nu(X)^{\kappa(X)sL_\theta(X)}$, and
(4) $|X|\le nu(X)^{\kappa(X)aL(X)}$.\\
In 1979, A. V. Arhangel{\cprime}ski{\u\i} asked if the inequality $|X|\le 2^{\chi(X)wL_c(X)}$ was true for every Hausdorff space $X$. It follows from the third inequality that the answer of this question is in the affirmative for all spaces with $nu(X)$ not greater than the cardinality of the continuum.\\
We also give a simple example of a Hausdorff space $X$ such that 
$|\cl_\theta(A)|>|A|^{\chi(X)}U(X)$ and 
$|\cl_\theta(A)|>(|A|\cdot U(X))^{\chi(X)}$, where $U(X)$ is the Urysohn 
number of $X$, recently introduced by Bonanzinga, Cammaroto and 
Matveev. This example shows that in (1) and (2) above, $nu(X)$ cannot 
be replaced by $U(X)$ and answers some questions posed by Bella and 
Cammaroto (1988), Bonanzinga, Cammaroto and Matveev (2011), and Bonanzinga 
and Pansera (2012).
\end{abstract}

\section{Introduction}

Let $X$ be a topological space and for $x\in X$ let $\mathcal{N}_x$ denote 
the family of all open neighborhoods of $x$ in $X$. For a nonempty subset $A$ 
of $X$ we denote by $\mathcal{U}_A$ the set of all families 
$\mathcal{U}=\{U_a: a\in A, U_a\in\mathcal{N}_a\}$ and by  
$\mathcal{C}_A$ the set of all families 
$\mathcal{C}=\{\overline{U}_a: a\in A, U_a\in\mathcal{N}_a\}$.

The \emph{$\theta$-closure} of a set $A$ in a space $X$ is the set 
$\cl_\theta(A) = \{x\in X :$ for every 
$U\in \mathcal{N}_x, \overline{U}\cap A \ne \emptyset\}$. $A$ is 
called \emph{$\theta$-closed} if $A=\cl_\theta(A)$ and $A$ is 
$\theta$-dense if $\cl_\theta(A)=X$ (see \cite{Vel66}). 
The smallest $\theta$-closed set containing $A$, i.e. the intersection 
of all $\theta$-closed sets containing $A$, is denoted by $[A]_\theta$ 
and is called \emph{the $\theta$-closed hull of $A$} \cite{BelCam88}.
The $\theta$-density of a space $X$ is 
$d_\theta(X):=\min\{|A|:A\subset X, \cl_\theta(A)=X\}$.

Recall that a space $X$ is called \emph{Urysohn} if every two distinct 
points in $X$ have disjoint closed neighborhoods. 

\begin{definition}[\cite{AlaKoc00}]
For a topological space $X$, $\kappa(X)$ is the smallest cardinal 
number $\kappa$ such that for each point $x\in X$, there is a 
collection $\mathcal{V}_x$ of closed neighborhoods of $x$ such that 
$|\mathcal{V}_x|\le \kappa$ and if $W\in \mathcal{N}_x$  
then $\overline{W}$ contains a member of $\mathcal{V} _x$.
\end{definition}

\begin{remark}
In \cite{AlaKoc00}, $\kappa(X)$ is defined only for Hausdorff spaces but clearly $\kappa(X)$ is well-defined for every topological space $X$. Also, an example of a Urysohn space $X$ is constructed in \cite{AlaKoc00} such that $\kappa(X)<\chi(X)$, where $\chi(X)$ is the character of the space $X$.
\end{remark}

\begin{definition}[\cite{BonCamMat11}] \label{D1}
The \emph{Urysohn number} of a topological space 
$X$, denoted by $U(X)$, is the smallest cardinal $\kappa$ such that for 
every $A\subset X$ with $|A|\ge \kappa$, there exists a family 
$\mathcal{C}\in\mathcal{C}_A$ such that 
$\cap\mathcal{C} =\emptyset$.
\end{definition}

Spaces $X$ with $U(X)=n$ for some integer $n\ge 2$ appeared first in \cite{BonCamMat07} and 
\cite{BonCamMatPan08} under the name \emph{$n$-Urysohn} and were studied further in \cite{BonPan12}. In \cite{CamCatPanTsa12} such spaces were called \emph{finitely-Urysohn}.

Clearly, $U(X)\ge 2$ and $U(X)\le |X|^+$ for every topological space 
$X$. If $X$ is Hausdorff then $U(X)\le |X|$ and $X$ is Urysohn if and 
only if $U(X)=2$ \cite{BonCamMat11}.

\section{On some questions related to the cardinality of the $\theta$-closed hull}\label{S2}

It was shown in \cite[Theorem 1]{BelCam88} that for every Urysohn space 
$X$, $|[A]_\theta|\le |A|^{\chi(X)}$ and the authors asked if that 
inequality holds true for every Hausdorff space (see 
\cite[Question]{BelCam88}). In \cite{BonCamMat11} the authors extended 
that result to all spaces with finite Urysohn number.

\begin{theorem}[{\cite[Proposition 4]{BonCamMat11}}]\label{TBCM}
For a set $A$ in a space $X$, if $U(X)$ is finite then 
$|[A]_\theta|\le |A|^{\chi(X)}$.
\end{theorem}

Since the proof given in \cite{BonCamMat11} does not apply for spaces 
with infinite Urysohn numbers the authors naturally asked the following 
question.

\begin{question}[{\cite[Question 5]{BonCamMat11}}]\label{QBCM1}
Is it true that for a set $A$ in a (assume Hausdorff if necessary) 
space $X$, $|[A]_\theta|\le |A|^{\chi(X)}U(X)$?
\end{question}

In \cite{BonPan12} the authors improved the inequality in 
Theorem \ref{TBCM} as follows and asked if even a stronger inequality 
than the one in Question \ref{QBCM1} holds true.

\begin{theorem}[{\cite[Proposition 7]{BonPan12}}]\label{TBP}
For a set $A$ in a space $X$, if $U(X)$ is finite then 
$|[A]_\theta|\le |A|^{\kappa(X)}$.
\end{theorem}

\begin{question}[{\cite[Question 9]{BonPan12}}]\label{QBCM2}
Is it true that for a set $A$ in a (Hausdorff) space $X$, 
$|[A]_\theta|\le |A|^{\kappa(X)}U(X)$?
\end{question}

The following example shows that the answer of Question \ref{QBCM1} 
(and therefore of the other two questions) is in the negative even for 
Hausdorff spaces with Urysohn numbers $U(X)=\omega$. 
Even more, our example shows that for Hausdorff spaces it is even possible that 
$|\cl_\theta(A)|>|A|^{\chi(X)}U(X)$ and 
$|[A]_\theta|>(|A|\cdot U(X))^{\chi(X)}$.
(For a different example see \cite[Example 3]{CamCatPanPor12}). 

\begin{example}\label{E1}
Let $\mathbb{N}$ denote the set of all positive integers, for 
$m\in \mathbb{N}$ let $\mathbb{N}_m:=\{n:n\in\mathbb{N}, n\ge m\}$, 
$\mathbb{R}$ be the set of all real numbers, and 
$\mathfrak{c}=|\mathbb{R}|$. Let 
also $S:=\{1/n: n\in \mathbb{N}\}\cup\{0\}$ and $\mathbb{N}\times S$
be the subspace of $\mathbb{R}\times\mathbb{R}$ 
with the inherited topology from $\mathbb{R}\times\mathbb{R}$. 
Let $\alpha$ be an initial ordinal and $\{B_\beta:\beta<\alpha\}$ be a 
family of $\alpha$ many pairwise disjoint copies of $\mathbb{N}\times S$.
We will refer to the points in $B_\beta$ as $(n,r)_\beta$, where 
$n\in \mathbb{N}$ and $r\in S$.
For each ordinal number $\beta<\alpha$, 
let $M_\beta:=B_\beta\cup \{\beta\}$ be the topological space with a 
topology such that $\{\beta\}$ is closed in $M_\beta$, all points in 
$B_\beta$ have the topology inherited from $\mathbb{R}\times\mathbb{R}$ 
and the point $\beta$ has as basic neighborhoods all sets of the form 
$\{\beta\}\cup\{\mathbb{N}_m\times (S\setminus\{0\})\}_\beta$, 
$m\in \mathbb{N}$. Now, let $X$ be the topological space 
obtained from the disjoint union of all spaces $M_\beta$, 
$\beta<\alpha$ after identifying for each $n\in\mathbb{N}$ all points 
of the form $(n,0)_\beta$, $\beta<\alpha$. We will denote those points 
by $(n,0)$. Then it is not difficult to verify that $X$ is a Hausdorff 
space (but not Urysohn) with Urysohn number $U(X)=\omega$, 
$\chi(X)=\kappa(X)=\omega$, and if $A$ is the subset 
$\{(n,0):n\in\mathbb{N}\}$ of $X$ then 
$|\cl_\theta(A)|=|[A]_\theta|=\alpha$ and 
$|A|^{\chi(X)}U(X)=|A|^{\kappa(X)}U(X)=\omega^\omega\cdot\omega=\mathfrak{c}$.
Therefore if $\alpha > \mathfrak{c}$ then we have 
$|\cl_\theta(A)|>|A|^{\chi(X)}U(X)$ and 
$|\cl_\theta(A)|>(|A|\cdot U(X))^{\chi(X)}$.
\end{example}

\section{Spaces with finite versus spaces with infinite Urysohn numbers}

We begin with the following lemma.

\begin{lemma}\label{LIG0}
Let $A$ be a nonempty subset of a topological space $X$ such that $\cap\mathcal{C}\ne\emptyset$ for every 
$\mathcal{C}\in\mathcal{C}_A$. Then 
$A\subset\cap\{\cl_\theta(\cap\mathcal{C}):\mathcal{C}\in\mathcal{C}_A\}$.
\end{lemma}
\begin{proof}
Let $\mathcal{C}_0\in\mathcal{C}_A$ and let $G=\cap\mathcal{C}_0\neq \emptyset$. 
Suppose that there exists $a_0\in A$ such that 
$a_0\notin \cl_\theta(G)$. Then there is $W_{a_0}\in \mathcal{N}_{a_0}$ 
such that $\overline{W}_{a_0}\cap G=\emptyset$. Let 
$\overline{V}_{a_0}:= \overline{U}_{a_0} \cap \overline{W}_{a_0}$, where 
$\overline{U}_{a_0}\in \mathcal{C}_0$ and $U_{a_0}\in \mathcal{N}_{a_0}$. 
Then the family $\mathcal{C}_1:=\{\overline{V}_{a_0}\}\cup \{\overline{U}_a:U_a\in\mathcal{C}_0,a\in A\setminus\{a_0\}\}$ has  
the property that $\cap\mathcal{C}_1=\emptyset$, a 
contradiction. Therefore $A\subset \cl_\theta(\cap\mathcal{C})$ for 
every $\mathcal{C}\in\mathcal{C}_A$, hence 
$A\subset\cap\{\cl_\theta(\cap\mathcal{C}):\mathcal{C}\in\mathcal{C}_A\}$.
\end{proof}

\begin{theorem}\label{TIG1}
Let $X$ be a topological space and $1<n<\omega$. Then $U(X)=n$  
if and only if there exists a set $A\subset X$ with $|A|=n-1$ such that 
$A=\cap\{\cl_\theta(\cap\mathcal{C}):\mathcal{C}\in\mathcal{C}_A\}$
and 
$\cap\{\cl_\theta(\cap\mathcal{C}):\mathcal{C}\in\mathcal{C}_B\}=\emptyset$ 
for every set $B$ with $|B|\ge n$. 
\end{theorem}
\begin{proof}
Suppose first that there exists a subset $A$ of $X$ with  
$|A|=n-1\ge 1$, such that 
$A=\cap\{\cl_\theta(\cap\mathcal{C}):\mathcal{C}\in\mathcal{C}_A\}$
and $\cap\{\cl_\theta(\cap\mathcal{C}):\mathcal{C}\in\mathcal{C}_B\}=\emptyset$ 
for every set $B$ with $|B|\ge n$. Then $A\subset \cl_\theta(\cap\mathcal{C})$ for every $\mathcal{C}\in\mathcal{C}_A$. Hence  $\cap\mathcal{C}\ne \emptyset$ for every 
$\mathcal{C}\in\mathcal{C}_A$ and therefore $U(X)>|A|=n-1$. Thus 
$U(X)\ge n$. Suppose that $U(X)>n$. Then there exists a set $B\subset X$ 
with $|B|= n$ such that $\cap\mathcal{C}\ne \emptyset$ for every 
$\mathcal{C}\in\mathcal{C}_B$. Then it follows from Lemma \ref{LIG0} 
that $B\subset\cap\{\cl_\theta(\cap\mathcal{C}):\mathcal{C}\in\mathcal{C}_B\}=\emptyset$, a contradiction. Therefore $U(X)=n$.

Now let $U(X)=n>1$. Then for every set $B\subset X$ such that $|B|=n$  
there exists $\mathcal{C}\in\mathcal{C}_B$ such that 
$\cap\mathcal{C}=\emptyset$ and therefore $\cap\{\cl_\theta(\cap\mathcal{C}):\mathcal{C}\in\mathcal{C}_B\}=\emptyset$. Also, there exists a set $A$ with $|A|=n-1$ such that 
for every $\mathcal{C}\in\mathcal{C}_A$ we have 
$\cap\mathcal{C}\ne\emptyset$. Then it follows from Lemma \ref{LIG0} 
that $A\subset\cap\{\cl_\theta(\cap\mathcal{C}):\mathcal{C}\in\mathcal{C}_A\}$.  To show that $A=\cap\{\cl_\theta(\cap\mathcal{C}):\mathcal{C}\in\mathcal{C}_A\}$, 
suppose that there is 
$x\in \cap\{\cl_\theta(\cap\mathcal{C}):\mathcal{C}\in\mathcal{C}_A\}\setminus A$. 
Therefore $\overline{U}\cap(\cap\mathcal{C})\ne\emptyset$ for every 
$U\in \mathcal{N}_x$ and every $\mathcal{C}\in\mathcal{C}_A$. 
Then for the set $B:=A\cup\{x\}$ we have that if 
$\mathcal{C'}\in\mathcal{C}_B$ then $\cap\mathcal{C'}\ne\emptyset$. 
Thus $U(X)>|B|=n$, a contradiction.
\end{proof}

\begin{remark}\label{R1}
Consider the sets $A:=\{(n,0):n<\omega\}\subset X$ and $\alpha\subset X$ in Example \ref{E1} and let $B_f$ and $B_i$ be a nonempty finite subset and an infinite subset of $\alpha$, respectively. If $\mathcal{C}\in\mathcal{C}_{B_f}$ then $\cap\mathcal{C}\subset A$ and $B_f\subsetneq\cap\{\cl_\theta(\cap\mathcal{C}):\mathcal{C}\in\mathcal{C}_{B_f}\}=\cl_\theta(A)=\alpha$, while there is $\mathcal{C}\in\mathcal{C}_{B_i}$ such that  $\cap\mathcal{C}=\emptyset$, hence $\cap\{\cl_\theta(\cap\mathcal{C}):\mathcal{C}\in\mathcal{C}_{B_i}\}=\emptyset$.
Therefore the space $X$ in Example \ref{E1} shows that Theorem \ref{TIG1} is not always valid when $U(X)$ is infinite even for Hausdorff spaces $X$, or in other words, the subsets in spaces with finite and infinite Urysohn numbers that determine the Urysohn number have different properties. Therefore we should not be surprised that theorems which are valid for spaces with finite Urysohn numbers are not necessarily valid for spaces with infinite Urysohn numbers (see Section \ref{S2}).
\end{remark}

The following two observations are valid for topological spaces with finite or infinite Urysohn numbers.

\begin{lemma}\label{LIG1}
Let $X$ be a topological space and $A$ be a nonempty subset of $X$. 
If $\cap\mathcal{C}\ne\emptyset$ for every 
$\mathcal{C}\in\mathcal{C}_F$ and every finite nonempty subset $F$ of $A$ then 
$A\subset\cap\{\cl_\theta(\cap\mathcal{C}):\mathcal{C}\in\mathcal{C}_F, \emptyset\ne F\subset A, |F|<\omega\}$.
\end{lemma}

\begin{proof}
Let $F$ be a nonempty subset of $A$, $\mathcal{C}_0\in\mathcal{C}_F$ and  $G=\cap\mathcal{C}_0$. Suppose that there exist $a_0\in A$ such that 
$a_0\notin \cl_\theta(G)$. Then there is $W_{a_0}\in \mathcal{N}_{a_0}$ 
such that $\overline{W}_{a_0}\cap G=\emptyset$. Let 
$\overline{V}_{a_0}:= \overline{W}_{a_0}$ if $a_0\notin F$ and  
$\overline{V}_{a_0}:= \overline{U}_{a_0} \cap \overline{W}_{a_0}$ if $a_0\in F$, where 
$\overline{U}_{a_0}\in \mathcal{C}_0$ and $U_{a_0}\in \mathcal{N}_{a_0}$. Then the family $\mathcal{C}_1:=\{\overline{V}_{a_0}\}\cup \{\overline{U}_a:U_a\in\mathcal{C}_0,a\in F\setminus\{a_0\}\}$ has  
the property that $\cap\mathcal{C}_1=\emptyset$, a 
contradiction. Therefore $A\subset \cl_\theta(\cap\mathcal{C})$ for 
every $\mathcal{C}\in\mathcal{C}_F$ and every nonempty finite subset $F$ of $A$, hence 
$A\subset\cap\{\cl_\theta(\cap\mathcal{C}):\mathcal{C}\in\mathcal{C}_F, \emptyset\ne F\subset A, |F|<\omega\}$.
\end{proof}

\begin{theorem}\label{TIG2}
Let $X$ be a topological space and $A$ be a nonempty subset of $X$. 
If $\cap\mathcal{C}\ne\emptyset$ for every 
$\mathcal{C}\in\mathcal{C}_F$ and every nonempty finite subset $F$ of $A$ then there exists a subset $M$ of $X$ such that $A\subset M$ and 
$M=\cap\{\cl_\theta(\cap\mathcal{C}):\mathcal{C}\in\mathcal{C}_F, \emptyset\ne F\subset M, |F|<\omega\}$.
\end{theorem}

\begin{proof}
Let $\alpha$ be an initial ordinal such that $\alpha=|X|^+$ and let $A$ 
satisfies the hypotheses of our claim. Then it follows from 
Lemma \ref{LIG1} that  
$A\subset\cap\{\cl_\theta(\cap\mathcal{C}):\mathcal{C}\in\mathcal{C}_F, \emptyset\ne F\subset A, |F|<\omega\}$. Suppose that there is 
$x_0\in \cap\{\cl_\theta(\cap\mathcal{C}):\mathcal{C}\in\mathcal{C}_F, \emptyset\ne F\subset A, |F|<\omega\}\setminus A$. 
Then $\overline{U}\cap(\cap\mathcal{C})\ne\emptyset$ for every 
$U\in \mathcal{N}_{x_0}$, every $\mathcal{C}\in\mathcal{C}_F$ and every nonempty finite subset $F$ of $A$. Then for the set $A_1:=A\cup\{x_0\}$ we have that if $F$ is a nonempty finite 
subset of $A_1$ and $\mathcal{C}\in\mathcal{C}_{F}$ then $\cap\mathcal{C}\ne\emptyset$.
Therefore, according to Lemma \ref{LIG1}, 
$A_1\subset\cap\{\cl_\theta(\cap\mathcal{C}):\mathcal{C}\in\mathcal{C}_F, \emptyset\ne F\subset A_1, |F|<\omega\}$. We continue this process for every 
$\beta<\alpha$ as follows. If $\beta=\gamma+1$ for some $\gamma<\alpha$ and $A_\gamma\subsetneq\cap\{\cl_\theta(\cap\mathcal{C}):\mathcal{C}\in\mathcal{C}_F, \emptyset\ne F\subset A_\gamma, |F|<\omega\}$ then we choose $x_\gamma\in \cap\{\cl_\theta(\cap\mathcal{C}):\mathcal{C}\in\mathcal{C}_F, \emptyset\ne F\subset A_\gamma, |F|<\omega\}\setminus A_\gamma$ and define $A_\beta:=A_\gamma\cup\{x_\gamma\}$. If $\beta$ is a limit ordinal then 
$A_\beta:=\cup\{A_\gamma:\gamma<\beta\}$. In that case it is clear that 
$\cap\mathcal{C}\ne\emptyset$ for every 
$\mathcal{C}\in\mathcal{C}_F$ and every nonempty finite subset $F$ of $A_\beta$ since every such $F$ is a subset of $A_\gamma$ for some $\gamma<\beta$. Therefore, according to Lemma \ref{LIG1}, we have 
$A_\beta\subset\cap\{\cl_\theta(\cap\mathcal{C}):\mathcal{C}\in\mathcal{C}_F, \emptyset\ne F\subset A_\beta, |F|<\omega\}$. If for some $\beta<\alpha$ we have  $A_\beta=\cap\{\cl_\theta(\cap\mathcal{C}):\mathcal{C}\in\mathcal{C}_F, \emptyset\ne F\subset A_\beta, |F|<\omega\}$ then we stop and take $M$ to be $A_\beta$. This process will eventually stop since for each $\gamma<\beta< \alpha$ we have $A_\gamma\subsetneq A_\beta\subseteq X$ and $\alpha > |X|$.
\end{proof}

\section{The non-Urysohn number of a space}

Motivated by the observations in the previous section we give the following definition.

\begin{definition}\label{DIG1}
A nonempty subset $A$ of a topological space $X$ is called 
\emph{finitely non-Urysohn} if for every nonempty finite subset $F$ of $A$ and 
every $\mathcal{C}\in\mathcal{C}_F$, $\cap\mathcal{C}\ne\emptyset$.
$A$ is called \emph{maximal finitely non-Urysohn subset of $X$} if $A$ is a 
finitely non-Urysohn subset of $X$ and if $B$ is a finitely non-Urysohn subset of $X$ such 
that $A\subset B$ then $A=B$.
\end{definition}

\begin{remark}\label{R2} {\rm (a)} Using Lemma \ref{LIG1} and Definition \ref{DIG1} one can easily verify that a nonempty subset $M$ of a topological space is maximal finitely non-Urysohn if and only if $M=\cap\{\cl_\theta(\cap\mathcal{C}):\mathcal{C}\in\mathcal{C}_F, \emptyset\ne F\subset M, |F|<\omega\}$.

{\rm (b)} It follows from Theorem \ref{TIG2} and Remark \ref{R2}(a) that every finitely non-Urysohn subset of a topological space is contained in a maximal one.

{\rm (c)} Using disjoint union of spaces as those constructed in Example \ref{E1} one can construct a Hausdorff topological space with (disjoint) maximal finitely non-Urysohn subsets with different cardinality. 

{\rm (d)} In a Urysohn space $X$ the only (maximal) finitely non-Urysohn subsets of $X$ are the singletons.
\end{remark}

Now we are ready to introduce the concept of a non-Urysohn number of a topological space $X$.

\begin{definition}\label{DIG2}
Let $X$ be a topological space. We define \emph{the non-Urysohn number 
$nu(X)$ of $X$} as follows: $nu(X):=1+\sup\{|M|:M$ is a (maximal) 
finitely non-Urysohn subset of $X\}$.
\end{definition}

\begin{remark}
It follows from Theorem \ref{TIG1} and Definition \ref{DIG2} 
that if $X$ is a topological space with a finite Urysohn number then 
$nu(X)=U(X)$. Also, $nu(X)\ge 2$ and $nu(X)\ge U(X)$ for every topological 
space $X$. For the space $X$ in Example \ref{E1}, $nu(X)=\alpha$ for 
every $\alpha$ while $U(X)=\alpha$ only if $\alpha<\omega$ and $U(X)=\omega$ if 
$\alpha\ge\omega$. Therefore there are even Hausdorff spaces $X$ for which  
$nu(X) >  U(X)$. 
\end{remark}

\section{On the cardinality of the $\theta$-closed hull}

In Theorem \ref{TIG3}, using the cardinal invariant non-Urysohn number of a space, we give an upper bound for $|\cl_\theta(A)|$ and $|[A]_\theta|$ of a subset $A$ in a topological space $X$. That theorem generalizes simultaneously all the results included in Theorem \ref{TT}. The proof of Theorem \ref{TIG3} follows proofs given in \cite{BelCam88}, \cite{AlaKoc00},  \cite{BonCamMatPan08}, \cite{BonCamMat11} or \cite{BonPan12}.

\begin{theorem}\label{TT}
Let $X$ be a space and $A\subset X$. 
\begin{itemize}
\item[(a)] If $X$ is Urysohn then 
$|[A]_\theta|\le |A|^{\chi(X)}$ \cite{BelCam88};
\item[(b)] If $X$ is Urysohn then 
$|\cl_\theta(A)|\le |A|^{\kappa(X)}$ \cite{AlaKoc00};
\item[(c)] If $U(X)$ is finite then 
$|[A]_\theta|\le |A|^{\chi(X)}$ \cite{BonCamMatPan08}, \cite{BonCamMat11};
\item[(d)] If $U(X)$ is finite then
$|[A]_\theta|\le |A|^{\kappa(X)}$ \cite{BonPan12}.
\end{itemize}
\end{theorem}

\begin{theorem}\label{TIG3}
Let $A$ be a subset of a topological space $X$. Then 
$|\cl_\theta(A)|\le |A|^{\kappa(X)}\cdot nu(X)$ and 
$|[A]_\theta|\le (|A|\cdot nu(X))^{\kappa(X)}$.
\end{theorem}

\begin{proof}
Let $\kappa(X) = m$, $nu(X) = u$, and $|A| = \tau$. For each $x\in X$ 
let $\mathcal{V}_x$ be a collection of closed neighborhoods of $x$ with  
$|\mathcal{V}_x|\le m$ and such that if $W$ is
a closed neighborhood of $x$ then $W$ contains a member of 
$\mathcal{V}_x$. For every $x\in \cl_\theta(A)$ and every 
$V\in \mathcal{V}_x$, fix a point $a_{x,V}\in V\cap A$, and let 
$A_x := \{a_{x,V}:V\in \mathcal{V}_x\}$. Let also
$\Gamma_x := \{V\cap A_x : V\in \mathcal{V}_x\}$. Then $\Gamma_x$ is a 
centered family (the intersection of any finitely many elements of 
$\Gamma_x$ is nonempty). It is not difficult to see that there are 
at most $\tau^m$ such centered families. Indeed $A_x\in [A]^{\le m}$ 
and $V\cap A_x \in [A]^{\le m}$, for every $V \in \mathcal{V}_x$. 
Since each centered family $\Gamma_x$ is a subset of $[A]^{\le m}$ and  
$|\Gamma_x|\le m$, the cardinality of the set of all such families is at 
most $(|A|^m)^m=|A|^m=\tau^m$.

We claim that the mapping $x\rightarrow \Gamma_x$ is $(\le u)$-to-one. Assume the contrary. Then there is a subset $K \subset \cl_\theta(A)$ such that $|K| = u^+$ and 
every $x\in K$ corresponds to the same centered family $\Gamma$. Since 
$nu(X)=u$, there exists 
a nonempty finite subset $F$ of $K$ and $\mathcal{C}\in \mathcal{C}_F$ 
such that $\cap{\mathcal{C}}=\emptyset$. Then for every $x \in F$ and 
$U_x\in \mathcal{C}$ we have $U_x \cap A_x \in \Gamma$; hence 
$\Gamma$ is not centered, a contradiction. 

Therefore the mapping $x \rightarrow \Gamma_x$ from $\cl_\theta(A)$ to 
$[[A]^{\le m}]^{\le m}$ is $(\le u)$-to-one, and thus 
\begin{equation}\label{Eq1}
|\cl_\theta(A)|\le u\cdot (\tau^m)^m = u\cdot \tau^m
\end{equation}
(Note that the proof that the 
mapping $x\rightarrow \Gamma_x$ is $(\le u)$-to-one does not 
depend upon the cardinality of the set $A$.)

It is not difficult to see (e.g., as in the proof of Theorem 1 in 
\cite{BelCam88}) that if we set $A_0 = A$ and 
$A_\alpha = \cl_\theta(\bigcup_{\beta<\alpha}A_\beta)$
for all $0<\alpha\le m^+$, then $[A]_\theta = A_{m^+}$. 
Let $\kappa=u\cdot \tau$. It follows from (\ref{Eq1}) that 
$|A_2|\le u\cdot (u\cdot \tau^m)^m=u^m\cdot\tau^m=(u\cdot \tau)^m=\kappa^m$.

To finish the proof we will show that if $\alpha$ is such that 
$2\le\alpha\le m^+$ then $|A_\alpha|\le \kappa^m$, and therefore 
$|[A]_\theta|\le \kappa^m$.

Suppose that $\alpha_0\le m^+$ is the smallest ordinal such that 
$|A_{\alpha_0}|> \kappa^m$. Then we have $|A_\beta|\le \kappa^m$
for each $\beta<\alpha_0$ and therefore 
$|\cup_{\beta<\alpha_0}A_\beta|\le\kappa^m\cdot m^+=\kappa^m$.
Now using (\ref{Eq1}) again we get 
$|A_{\alpha_0}|\le u\cdot(\kappa^m)^m =\kappa^m$, a contradiction.
\end{proof}

\begin{corollary}
If $X$ is a topological space then 
$|X|\le (d_\theta(X))^{\kappa(X)}nu(X)$.
\end{corollary}

\begin{remark}\label{R3}
Example \ref{E1} shows that the inequality $|\cl_\theta(A)|\le |A|^{\kappa(X)}\cdot nu(X)$ 
in Theorem \ref{TIG3} is exact. 
To see that, let $\alpha\ge\mathfrak{c}$ and take 
$A=\{(n,0):n\in\mathbb{N}\}\subset X$ and 
$M=\{\beta:\beta<\alpha\}\subset X$. Then $\cl_\theta(A)=A\cup M$ and 
therefore $|\cl_\theta(A)|=\alpha\le |A|^{\kappa(X)}\cdot nu(X)=\omega^\omega\cdot\alpha=\alpha$.

To see that the inequality  
$|[A]_\theta|\le (|A|\cdot nu(X))^{\kappa(X)}$ in Theorem \ref{TIG3}
is also exact, one can construct a Hausdorff (non-Urysohn) space $Y$ 
and a set $A\subset Y$ with $|[A]_\theta|=(|A|\cdot nu(Y))^{\kappa(Y)}$
as follows. Take the space $X_1:=X$ from Example \ref{E1} and the set 
$M=\{\beta:\beta<\alpha\}\subset X_1$. Represent the set $M$ as a disjoint 
union $\cup_{\beta<\alpha} M_\beta$ of countable infinite subsets of $M$. 
Take $\alpha$ many disjoint copies $\{X_\beta:\beta<\alpha\}$ of the space 
$X$ and for each $\beta<\alpha$ identify the points of $A_\beta$ with the 
points $\{(n,0):n<\omega\}\subset X_\beta$. Call the resulting space 
$X_2$. Now $X_1\subset X_2$ and in $X_2$ we have $\alpha$ many new copies 
of the set $M$. For each such set repeat the previous procedure to obtain 
the space $X_3$ and continue this procedure for each $n<\omega$. Call the 
resulting space $Y$. It is not difficult to see that $U(Y)=\omega$, 
$\chi(Y)=\kappa(Y)=\omega$, $nu(Y)=\alpha$, and if $A$ is the subset 
$\{(n,0):n\in\mathbb{N}\}$ of $X_1$ then 
$|[A]_\theta|=\alpha^\omega=(\omega\cdot\alpha)^\omega=(|A|\cdot nu(Y))^{\kappa(Y)}$.
Notice that if $\alpha>\mathfrak{c}$ is chosen to be a cardinal with a countable cofinality then $|[A]_\theta|=\alpha^\omega > \alpha=\omega^\omega\cdot\alpha= |A|^{\kappa(Y)}\cdot nu(Y)$ and therefore the right-hand side of the second inequality cannot be replaced by the right-hand side of the first inequality.
\end{remark}

\section{Some cardinal inequalities involving the non-Urysohn number}

We recall some definitions.

\begin{definition}[{\cite{WilDis84}, \cite{Hod06}}]
The \emph{almost Lindel\"{o}f degree of a space $X$ with respect to closed 
sets} is $aL_c(X):=\sup\{aL(F,X):F$ is a closed subset of $X\}$, where 
$aL(F,X)$ is the minimal cardinal number $\tau$ such that for every open 
(in $X$) cover $\mathcal{U}$ of $F$ there is a subfamily 
$\mathcal{U}_0\subset \mathcal{U}$ such that 
$|\mathcal{U}_0|\le \tau$ and 
$F\subset \cup\{\overline{U}:U\in\mathcal{U}_0\}$.
$aL(X,X)$ is called \emph{almost Lindel\"{o}f degree of $X$} and is 
denoted by $aL(X)$.
\end{definition}

\begin{remark}\label{R4}
The cardinal function $aL_c(X)$ was introduced in \cite{WilDis84} under 
the name Almost Lindel\"{o}f Degree and was denoted by $aL(X)$. Here we 
follow the notation and terminology used in \cite{Hod06} and suggested in 
\cite{BelCam88}. 
\end{remark}

\begin{definition}[\cite{Ala93}] 
The cardinal function $wL_c(X)$ is the smallest cardinal $\tau$
such that if $A$ is a closed subset of $X$ and $\mathcal{U}$ is an open 
(in $X$) cover of $A$, then there exists $\mathcal{V}\subset \mathcal{U}$ 
with $|\mathcal{V}|\le \tau$ such that 
$A\subset \overline{\cup\mathcal{V}}$.
\end{definition}

\begin{definition}[\cite{Arh95}] 
The cardinal function $sL(X)$ is the smallest cardinal $\tau$
such that if $A\subset X$ and $\mathcal{U}$ is an open (in $X$) cover of 
$\overline{A}$, then there exists $\mathcal{V}\subset \mathcal{U}$ with 
$|\mathcal{V}|\le \tau$ such that $A\subset \overline{\cup\mathcal{V}}$.
\end{definition}

\begin{definition}[\cite{AlaKoc00}] 
The cardinal function $sL_\theta(X)$ is the smallest cardinal $\tau$
such that if $A\subset X$ and $\mathcal{U}$ is an open (in $X$) cover of 
$\cl_\theta(A)$, then there exists $\mathcal{V}\subset \mathcal{U}$ with 
$|\mathcal{V}|\le \tau$ such that $A\subset \overline{\cup\mathcal{V}}$.
\end{definition}

Clearly $sL_\theta(X)\le sL(X)\le wL_c(X)\le L(X)$ and 
$aL(X)\le aL_c(X)\le L(X)$, where $L(X)$ is the Lindel\"{o}f degree of 
$X$. In \cite{AlaKoc00} an example of a Urysohn space $X$ is constructed 
such that $sL_\theta(X)<sL(X)$. For examples of Urysohn spaces such that 
$aL(X)<wL_c(X)$ or $wL_c(X)<aL(X)$ see \cite{Hod06} and for an example 
of a Urysohn space for which $aL(X)<aL_c(X)<L(X)$ see \cite{WilDis84} 
or \cite{Hod06}.

Here are some cardinal inequalities that involve the cardinal functions 
defined above. For more related results see the survey paper \cite{Hod06}.

\begin{theorem}\label{T}
\noindent
\begin{itemize}
\item[(a)] If $X$ is a Hausdorff space, then $|X|\le 2^{\chi(X)L(X)}$
\cite{Arh69}.
\item[(b)] If $X$ is a Hausdorff space, then $|X|\le 2^{\chi(X)aL_c(X)}$
\cite{WilDis84}.
\item[(c)] If $X$ is a Urysohn space, then $|X|\le 2^{\chi(X)aL(X)}$
\cite{BelCam88}.
\item[(d)] If $X$ is a Hausdorff space with $U(X)<\omega$, 
then $|X|\le 2^{\chi(X)aL(X)}$ \cite{BonCamMat11}.
\item[(e)] If $X$ is a Urysohn space, then 
$|X|\le 2^{\chi(X)wL_c(X)}$ \cite{Ala93}.
\item[(f)] If $X$ is a topological space with $U(X)<\omega$, then 
$|X|\le 2^{\chi(X)wL_c(X)}$ \cite{BonCamMat11}.
\item[(g)] If $X$ is a Hausdorff space, then 
$|X|\le 2^{\chi(X)sL(X)}$ \cite{Arh95}.
\item[(h)] If $X$ is a Urysohn space, then 
$|X|\le 2^{\kappa(X)sL_\theta(X)}$ \cite{AlaKoc00}.
\item[(i)] If $X$ is a topological space with $U(X)<\omega$, then 
$|X|\le 2^{\kappa(X)sL_\theta(X)}$ \cite{BonPan12}.
\end{itemize}
\end{theorem}

Recently in \cite{BonPan12}, after proving the inequality given in Theorem 
\ref{T}(i), the authors asked the following question. 

\begin{question}[{\cite[Question 11]{BonPan12}}]
Can one conclude that the inequality 
$$|X|\le U(X)^{\kappa(X)sL_\theta(X)}$$ is
true for every Hausdorff space $X$?
\end{question}

We show below that the answer of the above question is in the affirmative 
if the Urysohn number $U(X)$ is replaced by the non-Urysohn number 
$nu(X)$.

\begin{theorem}\label{TIG4}
For every topological space $X$, $|X|\le nu(X)^{\kappa(X)sL_\theta(X)}$.
\end{theorem}
\begin{proof}
Let $\kappa(X)sL_\theta(X)=m$ and $nu(X)=u$. For each $x\in X$ 
let $\mathcal{V}_x$ be a collection of closed neighborhoods of $x$ with  
$|\mathcal{V}_x|\le m$ and such that if $W$ is
a closed neighborhood of $x$ then $W$ contains a member of 
$\mathcal{V}_x$. Let $x_0$ be an arbitrary point in $X$.  
Recursively we construct a family 
$\{F_\alpha:\alpha<m^+\}$ of subsets of $X$ with the following 
properties:
\begin{itemize}
\item[(i)] $F_0=\{x_0\}$ and 
$\cl_\theta(\cup_{\beta<\alpha}F_\beta)\subset F_\alpha$ for 
every $0<\alpha<m^+$;
\item[(ii)] $|F_\alpha|\le u^m$ for every $\alpha < m^+$;
\item[(iii)] for every $\alpha<m^+$, and every 
$F\subset \cl_\theta(\cup_{\beta<\alpha}F_\beta)$ with $|F|\le m$ if 
$X\setminus \overline{\cup\mathcal{U}}\ne\emptyset$ for some 
$\mathcal{U}\in \mathcal{U}_F$, then
$F_\alpha\setminus \overline{\cup\mathcal{U}}\ne \emptyset$.
\end{itemize}

Suppose that the sets $\{F_\beta:\beta<\alpha\}$ satisfying (i)-(iii) have 
already been defined. We will define $F_\alpha$. Since $|F_\beta|\le u^m$ 
for each $\beta < \alpha$, we have 
$|\cup_{\beta<\alpha}F_\beta|\le u^m\cdot m^+=u^m$. 
Then it follows from Theorem \ref{TIG3}, that 
$|\cl_\theta(\cup_{\beta<\alpha}F_\alpha)|\le u^m$. 
Therefore there are at most $u^m$ subsets $F$ of 
$\cl_\theta(\cup_{\beta<\alpha}F_\alpha)$ with $|F|\le m$ and 
for each such set $F$ we have $|\mathcal{U}_F|\le m^m=2^m\le u^m$.
For each $F\subset\cl_\theta(\cup_{\beta<\alpha}F_\alpha)$ with $|F|\le m$
and each $\mathcal{U}\in\mathcal{U}_F$ for which 
$X\setminus \overline{\cup\mathcal{U}}\ne\emptyset$ we choose a point in
$X\setminus \overline{\cup\mathcal{U}}\ne\emptyset$ and let $E_\alpha$ be 
the set of all these points. Clearly $|E_\alpha|\le\ u^m$. Let 
$F_\alpha=\cl_\theta(E_\alpha\cup (\cup_{\beta<\alpha}F_\alpha))$. Then it follows from our 
construction that $F_\alpha$ satisfies (i) and (iii) while (ii) follows  
from Theorem \ref{TIG3}.

Now let $G=\cup_{\alpha<m^+}F_\alpha$. Clearly $|G|\le u^m\cdot m^+=u^m$. 
We will show that $G$ is $\theta$-closed. Suppose 
the contrary and let $x\in \cl_\theta(G)\setminus G$. Then for 
each $U\in\mathcal{V}_x$ we have $U\cap G\ne\emptyset$ and therefore there 
is $\alpha_U < m^+$ such that $U \cap F_{\alpha_U}\ne\emptyset$. Since 
$|\{\alpha_U: U\in \mathcal{V}_x\}|\le 
m$, there is $\beta< m^+$ such that $\beta > \alpha_U$ for every 
$U\in\mathcal{V}_x$ and therefore $x\in \cl_\theta(F_\beta)\subset G$, a 
contradiction.

To finish the proof it remains to check that $G = X$. Suppose that there 
is $x\in X\setminus G$. Then there is $V\in \mathcal{V}_x$ such that 
$V\cap G=\emptyset$. Hence, for every $y\in G$ there is
$V_y \in \mathcal{V}_y$ such that $V_y \cap \Int(V)=\emptyset$. Since 
$\{\Int(V_y) : y \in G\}$ is an open cover of $G$ and $G$ is 
$\theta$-closed, there is $F\subset G$ with $|F|\le m$ such
that $G \subset \overline{\cup_{y\in F}\Int(V_y)}$.
Since $|F|\le m$, there is $\beta<m^+$ such that $F\subset F_\beta$.
Then for $\mathcal{U}:= \{\Int(V_y) : y \in F\}$ we have 
$\mathcal{U}\in \mathcal{U}_F$ and 
$x\in X\setminus \overline{\cup\mathcal{U}}$.
Then it follows from our construction that 
$F_{\beta+1}\setminus \overline{\cup\mathcal{U}}\ne\emptyset$, a 
contradiction since $F_{\beta+1}\subset G\subset \overline{\cup\mathcal{U}}$.
\end{proof}

\begin{corollary}\label{CIG5}
For every topological space $X$, $|X|\le nu(X)^{\chi(X)wL_c(X)}$.
\end{corollary}

\begin{remark}
In 1979, A. V. Arhangel{\cprime}ski{\u\i} asked if the inequality $|X|\le 2^{\chi(X)wL_c(X)}$ was true for every Hausdorff topological space $X$ (see \cite[Question 2]{Hod06}). It follows immediately from Corollary \ref{CIG5} that the answer of his question is in the affirmative for all spaces with $nu(X)\le 2^\omega$. But as Example 3.10 in \cite{Got12b} shows, there are $T_0$-topological spaces for which that inequality is not true (in that example $nu(X)>2^{\omega}$).
\end{remark}

Modifying slightly the proof of Theorem \ref{TIG4} one can prove the 
following result.

\begin{theorem}
For every topological space $X$, $|X|\le nu(X)^{\kappa(X)aL(X)}$.
\end{theorem}

\begin{proof}
Let $\kappa(X)aL(X)=m$ and $nu(X)=u$. For each $x\in X$ 
let $\mathcal{V}_x$ be a collection of closed neighborhoods of $x$ with  
$|\mathcal{V}_x|\le m$ and such that if $W$ is
a closed neighborhood of $x$ then $W$ contains a member of 
$\mathcal{V}_x$. Let $x_0$ be an arbitrary point in $X$.  
Recursively we construct a family 
$\{F_\alpha:\alpha<m^+\}$ of subsets of $X$ with the following 
properties:
\begin{itemize}
\item[(i)] $F_0=\{x_0\}$ and 
$\cl_\theta(\cup_{\beta<\alpha}F_\beta)\subset F_\alpha$ for 
every $0<\alpha<m^+$;
\item[(ii)] $|F_\alpha|\le u^m$ for every $\alpha < m^+$;
\item[(iii)] for every $\alpha<m^+$, and every 
$F\subset \cl_\theta(\cup_{\beta<\alpha}F_\beta)$ with $|F|\le m$ if 
$X\setminus \cup\mathcal{C}\ne\emptyset$ for some 
$\mathcal{C}\in \mathcal{C}_F$, then
$F_\alpha\setminus \cup\mathcal{C}\ne \emptyset$.
\end{itemize}

Suppose that the sets $\{F_\beta:\beta<\alpha\}$ satisfying (i)-(iii) have 
already been defined. We will define $F_\alpha$. Since $|F_\beta|\le u^m$ 
for each $\beta < \alpha$, we have 
$|\cup_{\beta<\alpha}F_\beta|\le u^m\cdot m^+=u^m$. 
Then it follows from Theorem \ref{TIG3}, that 
$|\cl_\theta(\cup_{\beta<\alpha}F_\alpha)|\le u^m$. 
Therefore there are at most $u^m$ subsets $F$ of 
$\cl_\theta(\cup_{\beta<\alpha}F_\alpha)$ with $|F|\le m$ and 
for each such set $F$ we have $|\mathcal{C}_F|\le m^m=2^m\le u^m$.
For each $F\subset\cl_\theta(\cup_{\beta<\alpha}F_\alpha)$ with $|F|\le m$
and each $\mathcal{C}\in\mathcal{C}_F$ for which 
$X\setminus \cup\mathcal{C}\ne\emptyset$ we choose a point in
$X\setminus \cup\mathcal{C}\ne\emptyset$ and let $E_\alpha$ be 
the set of all these points. Clearly $|E_\alpha|\le\ u^m$. Let 
$F_\alpha=\cl_\theta(E_\alpha\cup (\cup_{\beta<\alpha}F_\alpha))$. Then it follows from our 
construction that $F_\alpha$ satisfies (i) and (iii) while (ii) follows  
from Theorem \ref{TIG3}.

Now let $G=\cup_{\alpha<m^+}F_\alpha$. Clearly $|G|\le u^m\cdot m^+=u^m$. 
We will show that $G$ is $\theta$-closed. Suppose 
the contrary and let $x\in \cl_\theta(G)\setminus G$. Then for 
each $U\in\mathcal{V}_x$ we have $U\cap G\ne\emptyset$ and therefore there 
is $\alpha_U < m^+$ such that $U \cap F_{\alpha_U}\ne\emptyset$. Since 
$|\{\alpha_U: U\in \mathcal{V}_x\}|\le 
m$, there is $\beta< m^+$ such that $\beta > \alpha_U$ for every 
$U\in\mathcal{V}_x$ and therefore $x\in \cl_\theta(F_\beta)\subset G$, a 
contradiction.

To finish the proof it remains to check that $G = X$. Suppose that there 
is $x\in X\setminus G$. Then there is $V\in \mathcal{V}_x$ such that 
$V\cap G=\emptyset$. Hence for every $y\in G$ there is
$V_y \in \mathcal{V}_y$ such that $V_y \cap \Int(V)=\emptyset$ and for 
every $z\in (X\setminus \{x\})\setminus G$  there is
$V_z \in \mathcal{V}_z$ such that $V_z \cap G=\emptyset$. Since 
$\{\Int(V_y) : y \in G\}\cup\{\Int(V_z) : z \in (X\setminus \{x\})\setminus G\}\cup\{\Int(V)\}$ is an open cover of $X$, there is $F'\subset X$ with $|F'|\le m$ such
that $X \subset {\cup_{t\in F'}V_t}$. Let $F:=F'\cap G\neq\emptyset$. 
Then $G\subset \cup\{V_y : y \in F\}$.
Since $|F|\le m$, there is $\beta<m^+$ such that $F\subset F_\beta$.
Then for $\mathcal{C}:= \{V_y : y \in F\}$ we have 
$\mathcal{C}\in \mathcal{C}_F$ and 
$x\in X\setminus \cup\mathcal{C}$.
Then it follows from our construction that 
$F_{\beta+1}\setminus \cup\mathcal{C}\ne\emptyset$, a 
contradiction since $F_{\beta+1}\subset G\subset \cup\mathcal{C}$.
\end{proof}

\begin{corollary}
For every topological space $X$ with $nu(X)<\omega$ (or, equivalently, 
$U(X)<\omega$), $|X|\le 2^{\kappa(X)aL(X)}$.
\end{corollary}

\begin{corollary}
For every Urysohn space $X$, $|X|\le 2^{\kappa(X)aL(X)}$.
\end{corollary}

\begin{remark}
In parallel with Definition \ref{DIG2} one can introduce the notion of a \emph{non-Hausdorff number} of a topological space. For results related to that notion see \cite{Got12b}. 
\end{remark}

\end{document}